\documentclass[12pt,reqno]{amsart}
\usepackage{amsmath}
\numberwithin{equation}{section}
\usepackage{hyperref}
\hypersetup{nesting=true,debug=true,naturalnames=true}
\usepackage{graphicx,amssymb,upref}
\usepackage{enumerate}
\usepackage[all]{xy}
\usepackage{color}
\usepackage{mathrsfs}
\usepackage{amsfonts}
\usepackage{caption}
\usepackage{paralist}
\usepackage{color}
\usepackage{float}
\usepackage{tikz,pgf}
\usetikzlibrary{calc}
\usetikzlibrary{intersections,decorations.markings}
\usepackage{tikz-cd}

\theoremstyle{definition}
\newtheorem{theorem}{\bf Theorem}[section]

\newtheorem{lemma}[theorem]{\bf Lemma}
\newtheorem{corollary}[theorem]{\bf Corollary}

\newtheorem{example}[theorem]{\bf Example}

\newtheorem{definition}[theorem]{\bf Definition}
\newtheorem{question}[theorem]{\bf Question}

\newtheorem{proposition}[theorem]{\bf Proposition}
{
\newtheorem*{Ack}{\bf Acknowledgements}
}
{

}
\newcommand{\mm}[1]{\mathrm{#1}}
\newcommand{\mb}[1]{\mathbb{#1}}
\makeatletter
\@namedef{subjclassname@2020}{\textup{2020} Mathematics Subject Classification}
\makeatother

\begin{document}

\title[On spherical fibrations and Poincare complexes]
{On spherical fibrations and Poincare complexes}

\author[Wen Shen]{Wen Shen}

\address{Department of Mathematics, Capital Normal University,
Beijing, P.R.China  }
\email{shenwen121212@163.com}

\begin{abstract}
In this paper, we prove that certain spherical fibrations over certain CW-complexes are stably fibre homotopy equivalent to $\mm{TOP}$-spherical fibrations (see Definition \ref{defiTOPfib}). Applying this result, we get a sufficient condition for whether a Poincar$\mm{\acute{e}}$ complex is of the homotopy type of a topological manifold. Moreover, we present the classification for some highly connected manifolds by the homotopy types of highly connected Poincar$\mm{\acute{e}}$ complexes.	
\end{abstract}

\subjclass[2020]{Primary 55R05, 57N16, 57P10}
\maketitle

\section{Introduction}

In the algebraic topology, the fibration is one of the fundamental objects. A fibration is a map $\xi = \{p : E \to  B\}$ which satisfies the following covering homotopy property: For every map $F : X \times I \to  B$ and every $g : X \to  E$ with $pg(x) = F (x, 0)$, there exists $G : X \times I \to  E$ with $G(x, 0) = g(x)$ and $pG = F$. 
From the classic homotopy theory, the fibers $F = p^{-1}(b)$ of a fibration $p : E\to B$ over each path
component of $B$ are all homotopy equivalent \cite{Hatcher}.   
If all fibers of a fibration are homotopy equivalent to $F$, we say it is an $F$-fibration. Two $F$-fibrations $\xi = \{p : E \to  B\}$, $\xi^\prime = \{p^\prime : E^\prime \to  B\}$ over the same space $B$ are called fibre homotopy equivalent if there are maps $f,g$ and homotopies $g\circ f\stackrel{H}{\simeq} id_E$, $f\circ g\stackrel{H^\prime}{\simeq} id_{E^\prime}$ so that $p^\prime \circ f=p$, $p \circ g=p^\prime$, and
\[
\xymatrix{
  E\ar[r]^-{{f}} \ar[d]^-{p}& E^\prime\ar[d]^-{{p}^\prime} & E^\prime\ar[r]^-{{g}} \ar[d]^-{p^\prime}& E\ar[d]^-{{p}}\\
{B}\ar@{=}[r]^-{} &B & B\ar@{=}[r]& B
}
\]
 $p\circ H$, $p^\prime\circ H^\prime$ are stationary homotopies (in other words, $H$ and $H^\prime$ only move points around in their fibers).

Another important object is the (locally trivial) fibre bundle $\xi=\{p:E\to B\}$ \cite{Hu1994}. 
For any $b\in B$, the spaces $p^{-1}(b)$ are homeomorphic to $F$ which is the fibre bundle's fiber. Moreover, every fibre bundle over a paracompact space (e.g., over a CW-complex) is a fibration. 

A principal $G$-bundle \cite{Steen1951} is a special fibre bundle where $G$ is a topological group.  
In \cite{Mil19561,Mil19562}, Milnor gave a construction of a universal principle bundle $\mathfrak{U}=({EG},p,{BG},G)$ for any topological group $G$ where $BG$ can be chosen to be a CW-complex. We say that ${BG}$ is the classifying space of the principal $G$-bundles. For any principal $G$-bundle over a compact space $B$, it admits a classifying map $f:B\to BG$ so that the pull-back bundle $f^\ast \mathfrak{U}$ is isomorphic to itself. For the morphisms between principle (fibre) bundles, one can refer to \cite{Hu1994}.

Every fibre bundle has a structure group $G$, thus associated with a principal $G$-bundle \cite{Hu1994}. Let $\xi$ be a fibre bundle with fiber $\mb{R}^n$. If its structure group is the group $\mm{O}_n$ of all orthogonal transformations of $\mb{R}^n$, then $\xi$ is the well-known vector bundle, also called an $\mm{O}_n$-bundle.
In particular, we can take the structure group of $\xi$ as the topological group $\mm{TOP}_n$ of homeomorphisms $(\mb{R}^n,0) \to  (\mb{R}^n,0)$ with the compact-open topology. Currently, we say that such bundle $\xi$ is a $\mm{TOP}_n$-bundle. Similarly to  vector bundle, every $\mm{TOP}_n$-bundle has a $0$-section.

For $F$-fibrations, there also exists a universal $F$-fibration $\gamma^F$ \cite{Rudyak}. The base space $B_F$ of a universal $F$-fibration is called a classifying space for $F$-fibrations, which can be chosen to be a CW-complex, and in this case for $F$-fibrations is uniquely defined up to homotopy equivalence.

 Every $n$-dimensional vector bundle is a $\mm{TOP}_n$-bundle by definition, every $\mm{T OP}_n$-bundle can be turned into an $S^{n-1}$-fibration by deleting the $0$-section. This hierarchy induces the sequence of forgetful maps
$\mm{BO}_n\to \mm{BTOP}_n\to \mm{B \mathcal G}_n$
where $\mm{B \mathcal G}_n$ is the classifying space for $S^{n-1}$-fibrations.
Indeed, there is another structure, i.e., piecewise linear (PL) $\mb{R}^n$-bundle \cite{KirbSieb1977} between $\mm{O}_n$-bundle and $\mm{TOP}_n$-bundle (see \cite{RoSa} about PL notions).

 Following lines of thought in the work of A. Grothendieck, Atiyah and Hirzebruch developed K-theory, which is a generalized cohomology theory defined by using stability classes of vector bundles. For $\mm{TOP}_n$-bundles, one can also define the product bundle and the Whitney sum, thus the same stability classes as vector bundles. 
 
 For $S^{n-1}$-fibration $\xi$ and $S^{m-1}$-fibration $\eta$ over the same base space $X$, we define the join $\xi *\eta$ \cite{Rudyak} which is an $S^{m+n-1}$-fibration over $X\times X$. Then we define the Whitney sum $\xi\oplus \eta= d^\ast(\xi * \eta)$ for spherical fibrations where $d : X \to  X \times X$ is the diagonal, and the stability classes of spherical fibrations. $\xi$ is stably isomorphic to $\eta$ if $\xi\oplus \theta^{N+m}$ is fibre homotopy equivalent to $\eta\oplus \theta^{N+n}$ for some $N$ where $\theta^{N}$ denotes the trivial $S^{N-1}$-fibration over $X$.

Let $B\mathcal V$ be the telescope of the sequence 
$$\cdots \to  B\mathcal{V}_{n-1} \to B\mathcal V_n\to  B\mathcal V_{n+1} \to \cdots $$
where $\mathcal V=\mm{O}$, $\mm{TOP}$, $\mathcal G$. For every finite CW-complex $X$, the set of all stable $\mathcal V$-objects over $X$ is in bijective correspondence with the homotopy class set $[X,B\mathcal V]$ \cite{Rudyak}, which is an abelian group. There also exists the sequence of forgetful maps
$$\mm{BO}\to \mm{BTOP}{\to } \mm{B \mathcal G}.$$
\begin{definition}\label{defiTOPfib}
	Let $\mathcal V=\mm{O}$ or $\mm{TOP}$. An $S^{n-1}$-fibration is called a $\mathcal V_n$-spherical fibration if it is realized by deleting the $0$-section of a $\mathcal V_n$-bundle.
	A spherical fibration is stably fibre homotopy equivalent to a $\mathcal V$-spherical fibration if it is stably isomorphic to a $\mathcal V_n$-spherical fibration for some $n$.
\end{definition}

From the definitions for $\mm{O}_n$-bundle, $\mm{TOP}_n$-bundle and $S^{n-1}$-fibration, there are natural questions as follows.
\begin{question}
	 Is an $S^{n-1}$-fibration fibre homotopy equivalent to a $\mathcal V_n$-spherical fibration?
\end{question}
\begin{question}
	 Is a spherical fibration stably fibre homotopy equivalent to a $\mathcal V$-spherical fibration?
\end{question}

In this paper, we consider whether a spherical fibration over a certain CW-complex is stably fibre homotopy equivalent to a TOP-spherical fibration. Any considered space is with a base point, and any map preserves the base points.

Let $X$ be a 1-connected CW-complex.  A stable spherical fibration $\xi$ over $X$ has a classifying map $X\to \mm{B\mathcal G}$ that admits a lift to the 2-fold cover $\mm{BS\mathcal G}$ of $\mm{B\mathcal G}$. Let the lift $f$ denote the classifying map.  
Note that $\xi$ is stably fibre homotopy equivalent to a $\mm{TOP}$-spherical bundle if the classifying map $f:X\to \mm{BS\mathcal G}$ has a homotopy lift to $\mm{BSTOP}$ where $\mm{BSTOP}$ is the 2-fold cover of $\mm{BTOP}$. 

Let $\mm{\mathcal {G}/TOP}$ be the homotopy fiber of $\mm{BSTOP}\to \mm{BS\mathcal G}$. By the works of Sullivan \cite{Sull1967} and Kirby-Siebenmann \cite{KirbSieb1977}, $\pi_n(\mm{\mathcal {G}/TOP})=0$ for $n=2i+1$, $\mb{Z}_2$ for $n=4i+2$, $\mb{Z}$ for $n=4i$.
Hence, by basic obstruction theory \cite{Hatcher} and the homotopy groups $\pi_\ast(\mm{\mathcal G/TOP})$, we have 
 \begin{proposition}\label{basicth}
Let $X$ be a finite $1$-connected CW-complex satisfying that $H^{4n+1}(X;\mb{Z})=0$ and $H^{4n+3}(X;\mb{Z}/2)=0$ for $n\ge 0$ where $\mb{Z}/2$ is integers mod $2$.
Then every spherical fibration over $X$ is stably fibre homotopy equivalent to a TOP-spherical fibration.	
\end{proposition}

However, we will see that some restrictions in Proposition \ref{basicth} on the cohomology groups of $X$ are excessive through the following example.
\begin{example}\label{ex}
	Let $f:S^{4k+1}\to \mm{BS\mathcal G}$ be the classifying map of any spherical fibration over $S^{4k+1}$, $k>0$. It is well-known that $\pi_n(\mm{BS}\mathcal G)$ is isomorphic to $\pi_{n-1}^s$, the $(n-1)$-dimensional stable homotopy group of the sphere for $n>1$ \cite{KirbSieb1977,Rudyak}, which is a finite abelian group by the work of Serre. From the following exact sequences
	$$\pi_{4k+1}(\mm{BSTOP})\to \pi_{4k+1}(\mm{BS\mathcal G})\to \pi_{4k}(\mathcal G/\mm{TOP})=\mb{Z}$$   
	 $f$ has a homotopy lift to $\mm{BSTOP}$.  
\end{example}

Since $\mm{\mathcal G/TOP}$ is an infinite loop space \cite{BoVo}, there is a space $\mm{B(\mathcal G/TOP)}$ such that $\mm{\mathcal G/TOP}\simeq \Omega \mm{B(\mathcal G/TOP)}$.
Thus there is a fibration sequence
$$\mm{\mathcal G/TOP}\to \mm{BSTOP}\to \mm{BS\mathcal G}\stackrel{\gamma}{\to }  \mm{B(\mathcal G/TOP)}$$
Applying fibration theory and local homotopy theory, we have the main theorem of this paper as follows.	  
\begin{theorem}\label{2nfree}
Let $X$ be a finite $1$-connected CW-complex satisfying that $H_{2n}(X;\mb{Z})$ is torsion-free for $n\ge 1$.
If the classifying map $f:X\to \mm{BS\mathcal G}$ of a spherical fibration over $X$  satisfies that
 $$(\gamma\circ f)_\ast:H_{2^{n+1}-1}(X;\mb{Z}/2)\to H_{2^{n+1}-1}(\mm{B(\mathcal G/TOP)};\mb{Z}/2)$$
  is trivial for $n\ge 1$, then the fibration is stably fibre homotopy equivalent to a TOP-spherical fibration.	
\end{theorem}

In particular, we have
\begin{corollary}\label{anyfibTOP}
	Let $X$ be a finite $1$-connected CW-complex satisfying that $H_{2n}(X;\mb{Z})$ is torsion-free for $n\ge 1$. If $H_{2^{n+1}-1}(X;\mb{Z}/2)=0$ for $n\ge 1$, then every spherical fibration over $X$ is stably fibre homotopy equivalent to a TOP-spherical fibration.
\end{corollary}

Next, we will give some applications of Corollary \ref{anyfibTOP} on the Poincar$\mm{\acute{e}}$ complex, which is one of the fundamental objects in surgery theory. 

 A CW-complex $X$ is a Poincar$\mm{\acute{e}}$ complex of dimension $m$, if there is an element $[X]\in H_m(X;\mb{Z})$ such that the cap product $$[X] \cap : H^q(X;\mb{Z})\to  H_{m-q}(X;\mb{Z})$$ is an isomorphism for all $q$. 
 The Poincar$\mm{\acute{e}}$ complex is a weaker structure than the well-known manifold. For Poincar$\mm{\acute{e}}$ complex, there exists a concept, ``Spivak normal fibration", which is analogous to the stable normal bundle \cite{Hu1994} of a smooth manifold.
 
 An $S^{k-1}$-fibration $\xi=\{\pi: E \to X\}$ is called a Spivak normal fibration \cite{Brow1972book} for a Poincar$\mm{\acute{e}}$ complex $X$ if there is $\alpha\in \pi_{m+k}(T(\xi))$ such that $h(\alpha)\cap U_\xi = [X]$, where $T(\xi)=X\cup_\pi C(E)$ is the mapping cone, $U_\xi\in H^k(T(\xi);\mb{Z})$ is the Thom class; $h: \pi_\ast\to H_\ast$ is the Hurewicz homomorphism.
 In \cite{Spi1967}, Spivak showed that if $\pi_1(X) = 0$ then a Poincar$\mm{\acute{e}}$ complex $X$ has a Spivak normal fibration, unique up to stable fibre homotopy equivalence.   
 Indeed, the Spivak normal fibration exists also, without any assumptions but Poincar$\mm{\acute{e}}$ duality with ordinary integer coefficients (see \cite{Brow1972}).
 
 From \cite{KirSie1969}, closed topological manifolds of dimension other than four are homeomorphic to CW-complexes (Poincar$\mm{\acute{e}}$ complexes in simply connected case), but the question is still open in dimension $4$. Conversely, 
 in differential topology and surgery theory, one studies the obstructions to lifting the structure of Poincar$\mm{\acute{e}}$ complex to the strong type found on a manifold. Novikov and Browder presented the following theorem independently.
 
 \begin{theorem}\label{Smooth}\cite{Brow1972book} \cite{Novi}
 	Let $X$ be a $1$-connected Poincar$\mm{\acute{e}}$ complex of dimension $m \ge  5$ with Spivak normal fibration $\eta$ which is stably fibre homotopy equivalent to an $\mm{O}$-spherical fibration. If 
 	\begin{itemize}
 		\item[(1)] $m$ is odd, or
 		\item[(2)] $m = 4k$ and $\mm{Index}X = (L_k(p_1 (\eta^{-1}),\cdots , p_k (\eta^{-1})))[X]$,
 	\end{itemize}
then there is a homotopy equivalence $f : M\to X$, $M$ a smooth $m$-manifold, such that $\nu = f^\ast(\eta)$, $\nu =$ normal bundle of $M \subset S^{m+k}$. \end{theorem}
 
  The above theorem can be extended to the topological case by  
 \cite{KirSie1969}.
 \begin{theorem}\cite{Brow1971}\label{PLTOP} 
 Let $X$ be a $1$-connected Poincar$\mm{\acute{e}}$ complex of dimension $m \ge  5$. Then $X$ is of the homotopy type of a topological manifold if and only if its Spivak normal fibration is stably fibre homotopy equivalent to a TOP-spherical fibration. 
  \end{theorem}

Due to Theorem \ref{Smooth} and \ref{PLTOP}, the problem of making a general spherical fibration into an $\mm{O}$ or $\mm{TOP}$-spherical fibration (up to stable fibre homotopy equivalence) is central to the theory of surgery.

By Corollary \ref{anyfibTOP} and Theorem \ref{PLTOP}, we get
\begin{theorem}\label{homotomanifold}
	If a finite $1$-connected Poincar$\mm{\acute e}$ complex $X$ of dimension $m \ge  5$ satisfies that (1) $H_{2n}(X)$ is torsion free for $n\ge 1$.
\begin{itemize}
\item[(2)] $ H_{2^{n+1}-1}(X;\mb{Z}/2)=0$ for $n\ge 1$.	
\end{itemize}
Then $X$ is of the homotopy type of a topological manifold.
\end{theorem}

From Theorem \ref{homotomanifold}, a finite $(n-1)$-connected Poincar$\mm{\acute e}$ complex $X$ of dimension $2n$ is of the homotopy type of a topological manifold if $n\ne 2^j-1$. By  surgery theory, we have
\begin{corollary}\label{onetoone}
	Let $n$ be odd so that $n\ne 2^j-1$.
The homotopy types of finite $(n-1)$-connected Poincar$\mm{\acute e}$ complexes of dimension $2n$ are in one-to-one correspondence with the homeomorphism types of $(n-1)$-connected closed topological manifolds of dimension $2n$. 
\end{corollary}

For a certain odd number $n$ so that $\pi_{n-1}^s=0$, one can also define the Kervaire invariant for finite $(n-1)$-connected Poincar$\mm{\acute e}$ complexes $P$ of dimension $2n$ 
$$\Phi(P)\in \mb{Z}/2$$
following the construction of Kervaire in \cite{Kerva1960}.  

\begin{theorem}\label{decthomotopytype}
	Let $n$ be odd so that $\pi_{n-1}^s=0$.
	\item (1) The homotopy types of $(n-1)$-connected Poincar$\mm{\acute e}$ complexes of dimension $2n$ are determined by the $n$-th Betti number and the Kervaire invariant. 
	
\item (2) The Poincar$\mm{\acute e}$ complex of Kervaire invariant one always exists for any $n$ satisfying the condition and any value of the $n$-th Betti number. 
\end{theorem}

For example, in case of $n=5$, $13$, $\pi_{n-1}^s=0$ \cite{Toda}.

For such $n$ as in Theorem \ref{decthomotopytype}, let $P$ be an $(n-1)$-connected Poincar$\mm{\acute e}$ complex of dimension $2n$ with $H^n(P;\mb{Z})=\mb{Z}\oplus \mb{Z}$. If the Kervaire invariant of $P$ equals to $0$, $P$ is homotopy equivalent to $S^n\times S^n$.

By Corollary \ref{onetoone} and Theorem \ref{decthomotopytype}, we also have a classification for $(n-1)$-connected topological manifolds of dimension $2n$.
\begin{corollary}
	Let $n\ne 2^j-1$ be odd so that $\pi_{n-1}^s=0$. 
\item (1) The homeomorphism types of $(n-1)$-connected closed  topological manifolds of dimension $2n$ are determined by the $n$-th Betti number and the Kervaire invariant. 
	
\item (2) The closed topological manifold of Kervaire invariant one always exists for any $n$ satisfying the condition and any value of the $n$-th Betti number. 
	In particular, it does not admit any smooth structure when additionally $n\ne 1$ mod $8$.
\end{corollary}

The above corollary presents an answer to the question raised by Kervaire in \cite{Kerva1960}.

The plan of this paper is as follows.
We first introduce some preliminary knowledge about the local homotopy theory in Section \ref{Seclocalhomo}, then prove a preparatory lemma in Section \ref{Sechomoonhomology}, which reduces the proof of Theorem \ref{2nfree} to the computations for the homomorphism 
$$(\gamma\circ f)_\ast:H_\ast(X;G)\to H_\ast(\mm{B(\mathcal G/TOP)};G)$$
where $G=\mb{Z}/2$ or $\mb{Q}$.
 These computations are divided into two cases, shown in Sections \ref{SecZ2homo} and \ref{SecQhomo} respectively. Next, we recall the surgery exact sequence to prove Corollary \ref{onetoone} in Section \ref{SecSurgery}. At last, we investigate the classification for certain highly connected Poincar$\mm{\acute e}$ complexes and topological manifolds in Section \ref{classifsec}, and detect the smooth structure in Section \ref{smoothsec}.

\section{Local homotopy theory}\label{Seclocalhomo} 
 In the Sullivan-Quillen proof of the Adams conjecture on the image of the $J$-homomorphism, and in Sullivan's work on BPL and $\mm{\mathcal G/PL}$ (see \cite{KirbSieb1977,MaMi}), it has become necessary to systematically exclude $p$-primary information about CW-complexes for certain primes $p\ge 2$. In this respect, the local homotopy theory is a powerful tool. We only present fundamental definitions and conclusions in this section. For more details, one can refer to \cite{MayPon2012} \cite{Sull2005}.    

We first discuss some algebraic constructions.
Let $T$ be a set of primes in the ring $\mb{Z}$ of integers. We will write ``$\mb{Z}$ localized at $T$" \cite{Hun}
$$\mb{Z}_{(T)}=S^{-1}\mb{Z}$$
where $S$ is the multiplicative set generated by the primes not in $T$.
\begin{definition}
	If $G$ is an abelian group then the localization of $G$ with respect to a set of primes $T$, $G_{(T)}$ is the $\mb{Z}_{(T)}$-module $G\otimes \mb{Z}_{(T)}$ where the right $\mb{Z}_{(T)}$-action is on the right factor $\mb{Z}_{(T)}$ by the multiplication. The ``localization homomorphism $\phi:G\to G_{(T)}$" is $\phi(g)=g\otimes 1$.
\end{definition}
$G_{(T)}$ is also an abelian group. An abelian group $G$ is said to be ``$T$-local" if it admits a structure of $\mb{Z}_{(T)}$-module, necessarily unique. The localization homomorphism $\phi$ is universal, which means that for any homomorphism $f : G \to  A$, where $A$ is $T$-local, there is a unique homomorphism $\tilde f $ that makes the following diagram commute.
 \[
\xymatrix@C=0.8cm{
G\ar[rd]_-{f}\ar[rr]^-{\phi}&&G_{(T)}\ar@{.>}[ld]^-{\tilde f}\\
  & A&  
}
\]

Next, we define a localization of spaces.
\begin{definition}
 $X_T$ is a $T$-local space iff $\pi_\ast X_T$ is $T$-local. For any space $X$, a map of $X$ into a local space $X_T$, $\phi:X\to X_T$, is a localization of $X$ if it is universal, i.e. given $f : X \to  Z$ from $X$ to a  $T$-local space $Z$, then there is a map $\tilde{ f}$, unique up to homotopy, that makes the following diagram commute up to homotopy.
 \[
\xymatrix@C=0.8cm{
X\ar[rd]_-{f}\ar[rr]^-{\phi}&&X_T\ar@{.>}[ld]^-{\tilde f}\\
  & Z&  
}
\]
\end{definition} 

\begin{theorem}\cite{Sull2005}\label{localexist}
	Any 1-connected space has a localization.
\end{theorem}

In \cite{MayPon2012}, the above theorem is proved in the category of nilpotent spaces. Here, we only consider 1-connected spaces.

There are equivalent descriptions for local spaces and localization.
\begin{theorem}\cite{Sull2005}\label{equilocal} For a map $\phi: X \to  X^\prime$ of arbitrary 1-connected spaces, the following are equivalent \item (1) $\phi$ is a localization.\item (2) $\phi$ localizes integral homology groups.
\[\xymatrix@C=0.8cm{\tilde H_\ast(X)\ar[rd]_-{\otimes 1}\ar[rr]^-{\phi_\ast}&&\tilde H_\ast(X^\prime)\ar[ld]^-{\cong}\\& \tilde H_\ast(X)\otimes \mb{Z}_{(T)}&  }\]
\item (3) $\phi$ localizes homotopy groups.
\[\xymatrix@C=0.8cm{\pi_\ast(X)\ar[rd]_-{\otimes 1}\ar[rr]^-{\phi_\ast}&&\pi_\ast(X^\prime)\ar[ld]^-{\cong}\\&\pi_\ast(X)\otimes \mb{Z}_{(T)}&  }\]
\end{theorem}

Finally, we discuss the $2$-local homotopy type of $\mm{B(\mathcal G/TOP)}$.
By the homotopy groups of $\mm{\mathcal {G}/TOP}$, we have $\pi_n(\mm{B(\mathcal {G}/TOP}))=0$ for $n=2i$, $\mb{Z}_2$ for $n=4i-1$, $\mb{Z}$ for $n=4i+1$. 

Let $T=\{2\}$,
 $\mm{B(\mathcal G/TOP)}[2]=\mm{B(\mathcal G/TOP)}_{T}$ be the $2$-localization of $\mm{B(\mathcal G/TOP)}$. 
   By \cite{MaMi}, $$\mm{B(\mathcal G/TOP)}[2]\simeq \Pi_{k=1}^\infty \mm{K}(\mb{Z}_{(2)},4k+1)\times \mm{K}(\mb{Z}/2,4k-1).$$
\begin{lemma}\label{SktoBGTOP}
	 A map $\iota: S^{n}\to \mm{B(\mathcal G/TOP)}$ is trivial if and only if 
	 $$\iota_\ast:H_n(S^n;G)\to H_n(\mm{B(\mathcal G/TOP)};G)$$
 is trivial where $G=\mb{Z}$ for $n=1\mod 4$, $\mb{Z}/2$ for $n=-1\mod 4$.
\end{lemma}
\begin{proof}
The lemma follows by applying the Hurewicz homomorphism on the composition of the map $\iota$ and $\phi:\mm{B(\mathcal G/TOP)}\to \mm{B(\mathcal G/TOP)}[2]$.
\end{proof} 

For dimension $4k+1$, we have a more precise lemma.
\begin{lemma}\label{homotopy4njia1}
	 A map $\iota: S^{4k+1}\to \mm{B(\mathcal G/TOP)}$ is nontrivial if and only if its image $\mm{H}(\iota)$ equals to $sl_{4k+1}+R\in H_{4k+1}(\mm{B(\mathcal G/TOP)};\mb{Z})$ under the Hurewicz homomorphism $\mm{H}$, where $0\ne s\in \mb{Z}$, $l_{4k+1}$ is a free generator, $R$ denotes other components. Moreover, under the composite map of the induced homomorphism by the $2$-localization and the projection  
	$$ H_\ast(\mm{B(\mathcal G/TOP)};\mb{Z})\stackrel{\phi_\ast}{\to} H_\ast(\mm{B(\mathcal G/TOP)}[2];\mb{Z})\to H_\ast(\mm{K}(\mb{Z}_{(2)},4n+1);\mb{Z})$$  
	the image of $R$ is zero, the image of $l_{4k+1}$ corresponds to a generator of $H_{4k+1}(\mm{K}(\mb{Z}_{(2)},4k+1);\mb{Z})$.
\end{lemma}
\begin{proof}
	The lemma follows by applying the Hurewicz homomorphism on the map $S^{4n+1}\to \mm{B(\mathcal G/TOP)}\to \mm{B(\mathcal G/TOP)}[2]\to \mm{K}(\mb{Z}_{(2)},4n+1)$.
\end{proof}

Summarizing Lemma \ref{SktoBGTOP} and \ref{homotopy4njia1}, we have
\begin{theorem}\label{homotopyhomology}
	A map $\iota: S^{n}\to \mm{B(\mathcal G/TOP)}$ is trivial if and only if 
	 $$\iota_\ast:H_n(S^n;G)\to H_n(\mm{B(\mathcal G/TOP)};G)$$
 is trivial where $G=\mb{Q}$ for $n=1\mod 4$, $\mb{Z}/2$ for $n=-1\mod 4$.
\end{theorem}

\section{The homomorphisms on homology groups}\label{Sechomoonhomology}
Recall the fibration sequence
$$\mm{\mathcal G/TOP}\to \mm{BSTOP}\stackrel{\mathcal F}{\to } \mm{BS\mathcal G}\stackrel{\gamma}{\to} \mm{B(\mathcal G/TOP)}$$ 
\begin{lemma}\label{liftlemma}
	For CW-complex $Y$, a map $f:Y\to \mm{BS\mathcal G}$ has a homotopy lift to $\mm{BSTOP}$ if and only if 
	$\gamma \circ f$ is homotopic to the constant map. 
\end{lemma}
\begin{proof} 
Since $\mm{\mathcal G/TOP}$ is an infinite loop space \cite{BoVo}, it is an associative $H$-space. By \cite{DL}, there is a universal principle fibration $$\mathcal U= \{\mathfrak E,p,\mathfrak B,\mm{\mathcal G/TOP}\}$$ such that $\mathfrak B=\mm{B(\mathcal G/TOP)}$ and $\pi_i(\mathfrak E)=0$ for $i\ge 0$. 
 
Let the pull-back fibration of $\gamma$ be
$$\gamma^\ast \mathcal U= \{\mathfrak {\bar {E}},\bar p,\mm{ BS\mathcal G},\mm{\mathcal G/TOP}\}.$$
The diagram of the pull-back fibration is 
 a homotopy pull-back \cite{Math1976}. Since $\mm{BSTOP}$ is the homotopy fiber of the map $\gamma$, $\gamma \circ \mathcal F$ is homotopic to the constant map $c$. There is a map $\mathcal H: \mm{BSTOP} \to \mathfrak {\bar {E}}$ such that $\bar p\circ \mathcal H \simeq \mathcal F$ by the property of homotopy pull-back. 
\[
\xymatrix@C=0.8cm{
\mm{BSTOP}\ar@/^/[rrd]^c \ar@/_/[drd]_-{\mathcal F}\ar[dr]^-{\mathcal H}&&\\
 &\mathfrak {\bar {E}}\ar[d]^-{\bar p}\ar[r]^-{\bar \gamma} &\mathfrak E\ar[d]^-{p}\\
   &  \mm{ BS\mathcal G}\ar[r]^-{\gamma}&\mm{B(\mathcal G/TOP)}
}\]
By the exact sequences of homotopy groups of homotopy fibrations, $\mathcal H$ is a weak homotopy equivalence. 
\[
\xymatrix@C=0.8cm{
E\ar[r]\ar[d]^-{}&\mathfrak {\bar {E}}\ar[d]^-{}\ar[r]&\mathfrak E \ar[d]\\
Y  \ar[r]^-{f} &  \mm{BS\mathcal G}\ar[r]^-{\gamma}&\mm{B(\mathcal G/TOP)}
}
\] 

By the given map $f$, we have an induced principal fibration over $Y$. If $\gamma\circ f$ is homotopic to the constant map, then $E$ is homotopy equivalent to $\mm{\mathcal G/TOP}\times Y$. So $f$ has a lift to $\mathfrak {\bar {E}}$.

Note that $Y$ is a CW-complex. For a weak homotopy equivalence $w:\mathcal A\to\mathcal B$ between two spaces $\mathcal A,\mathcal B$, the induced map $[Y,\mathcal A]\to [Y,\mathcal B]$ is a one-to-one correspondence \cite{WhiteGW1978}. Then $f$ has a lift $\bar f$ to $\mm{BSTOP}$ such that $f\simeq \mathcal F \circ \bar f$.

Conversely, assume that $f$ has a lift $\bar f$ to $\mm{BSTOP}$ such that $f\simeq \mathcal F \circ \bar f$. Since $\gamma \circ \mathcal F$ is homotopic to the constant map, $\gamma \circ f$ is also.
\end{proof}

Let $X$ be a finite 1-connected CW complex.  Every stable spherical fibration over $X$ has a classifying map $f:X\to \mm{BS\mathcal G}$.
Let $X_n$ be the $n$-dimensional skelton of $X$, $f_n:X_n\to \mm{BS\mathcal G}$ be the restriction of $f$ on $X_n$. 
In this section, we assume that $X$ satisfies
\begin{equation}
	\text{$H_{2k}(X)$ is torsion-free for $k\ge 1$.} \label{torsionfree02}
\end{equation}

By (\ref{torsionfree02}) and simplest cell structure \cite{Wall1965}, $X_2=\vee_{i\in T^2} S_i^2$ where $T^2$ is a finite set or empty set. Since 
$$\pi_2(\mm{B(\mathcal G/TOP)})=\pi_1(\mm{\mathcal G/TOP})=0,$$ 
 $\gamma\circ f_2:X_2\to \mm{B(\mathcal G/TOP)}$ is homotopic to the constant map.

Assume $\gamma\circ f_n$ is homotopic to the constant map, i.e., $\gamma\circ f_n\stackrel{H}{\simeq} c$. Let $$X_{n+1}= X_n\cup_{\beta_i,i\in T^n} e^{n+1}$$ where $T^n$ is a finite set or empty set, $\beta_i$ is an attaching map. Since the inclusion $X_n\to X_{n+1}$ is a cofibration \cite{Hatcher}, we have a homotopy extension $\bar H$ such that $\bar H_0=\gamma\circ f_{n+1}$, $\bar H_1|_{X_n}=c$.
 \[
\xymatrix@C=0.8cm{
X_{n}\times I\cup X_{n+1}\times \{0\}\ar[d]\ar[rr]^-{H\cup \gamma\circ f_{n+1}}&&\mm{B(\mathcal G/TOP)}\\
X_{n+1}\times I\ar@{.>}[rru]_-{\bar{H}}& & 
}
\]

Thus we have the following commutative diagram up to homotopy where $f\circ \iota=f_{n+1}$. 
\[
\xymatrix@C=0.8cm{
X_{n+1}\ar[d]^-{\mathfrak p}\ar[r]^-{\iota}&X\ar[r]^-{f}&\mm{BS\mathcal G}\ar[d]^-{\gamma}\\
X_{n+1}/X_n=\vee_{i\in T^n} S_i^{n+1}\ar[rr]^-{\bar{\iota}}&  &   \mm{B(\mathcal G/TOP)}
}
\]

If $n$ is odd, $\bar{\iota}$ is homotopic to the constant map by $$\pi_{2i}(\mm{B(\mathcal G/TOP)})=\pi_{2i-1}(\mm{(\mathcal G/TOP)})=0.$$ Thus $\gamma\circ f_{n+1}$ is homotopic to the constant map. 

Next, we discuss the case for $n=2k$.
Recall the assumption (\ref{torsionfree02}), i.e, $H_{2k}(X)$ is torsion free for $k\ge 1$. By the simplest cell structure [\cite{Wall1965}, Proposition 4.1], we have
$$H_{2k}(X_{2k};\mb{Z})\cong  H_{2k}(X_{2k+1};\mb{Z})\cong  H_{2k}(X;\mb{Z}).$$ 
Thus, by the long exact sequence for the pair $(X_{2k+1},X_{2k})$,
\begin{equation}
	H_{2k+1}(X_{2k+1};\mb{Z})\cong  H_{2k+1}(X_{2k+1}/X_{2k};\mb{Z})=H_{2k+1}(\vee S^{2k+1};\mb{Z}) \label{X2k1andS2k1}
\end{equation}
By the above homotopy commutative diagram,
$$(\gamma\circ f_{2k+1})_\ast=(\bar{\iota}\circ \mathfrak p)_\ast:H_{2k+1}(X_{2k+1};G)\to H_{2k+1}(\mm{B(\mathcal G/TOP)};G)$$
for $G=\mb{Q}$ and $\mb{Z}/2$.  
If the homomorphisms
\begin{equation}
	(\gamma\circ f)_\ast:H_{4i+1}(X;\mb{Q})\to H_{4i+1}(\mm{B(\mathcal G/TOP)};\mb{Q}) \label{4njia1}
	\end{equation}
	\begin{equation}
	(\gamma\circ f)_\ast:H_{4i+3}(X;\mb{Z}/2)\to H_{4i+3}(\mm{B(\mathcal G/TOP)};\mb{Z}/2) \label{4njia3}
\end{equation}
 are trivial for any $i\ge 0$, then $\bar{\iota}$ is homotopic to the constant map by Theorem \ref{homotopyhomology}. Thus $\gamma\circ f_{n+1}$ is homotopic to the constant map.
By induction and Lemma \ref{liftlemma}, $f$ has a homotopy lift to $\mm{BSTOP}$.

In the next sections, we compute the homomorphisms (\ref{4njia1}) and (\ref{4njia3}).

\section{The $\mb{Z}/2$-homology homomorphism}\label{SecZ2homo}
In this section, the Serre spectral sequence \cite{Sw} is used frequently. Now we first give a fundamental proposition for it.
\begin{proposition}\cite{Sw}\label{edgehomo}
	For the Serre spectral sequence
	$$E_2^{p,q}=H_p(B;H_q(F;\mb{Z}/2))\Longrightarrow H_{p+q}(E;\mb{Z}/2)$$
	of the fibration $F\to E\to B$ where $B$ is connected, 
	\item (1) $E_\infty^{p,0}$ is a subgroup of $H_p(B;\mb{Z}/2)$ and equals to $$\mm{im}\{H_p(E;\mb{Z}/2)\to H_p(B;\mb{Z}/2)\};$$ 
	\item (2) $E_\infty^{0,q}$ is a quotient of $H_p(F;\mb{Z}/2)$ and isomorphic to $$\mm{im}\{H_p(F;\mb{Z}/2)\to H_p(E;\mb{Z}/2)\}.$$ 
\end{proposition} 

For the spectral sequence on cohomology, there also exists a dual version of the above proposition.

Let $X$ be a finite 1-connected CW complex,  
 $f:X\to \mm{BS}\mathcal G$ be a classifying map of a stable spherical fibration over $X$. Then we focus on the following homomorphism 
\begin{equation}
	(\gamma\circ f)_\ast:H_{2^{j+1}-1}(X;\mb{Z}/2)\to H_{2^{j+1}-1}(\mm{B(\mathcal G/TOP)};\mb{Z}/2)\label{gammaf}
\end{equation}

Let $\mb{K}=\Pi_{k=1}^\infty \mm{K}(\mb{Z}_{(2)},4k)\times \mm{K}(\mb{Z}/2,4k-2)$. Let $\mm{\mathcal G/TOP}[2]$ be the $2$-localization of $\mm{\mathcal G/TOP}$. 
It is well known that \cite{MaMi} $$\mm{\mathcal G/TOP}[2]\simeq \mb{K}.$$ 
Since localization preserves fibrations in the category of 1-connected spaces \cite{Sull2005}, we have the commutative diagram for fibrations.
\[
\xymatrix@C=0.8cm{
\mm{\mathcal G/TOP}\ar@{=}[r]\ar[d]^-{\varphi}&\mm{\mathcal G/TOP}\ar@{=}[r]\ar[d]^-{\mathfrak q}&\mm{\mathcal G/TOP}\ar[d]^-{}\ar[r]&\mm{\mathcal G/TOP}[2]\ar[d]\\
E\ar[r]\ar[d]^-{}&\bar{\mathfrak E}\ar[r]\ar[d]^-{}&\mathfrak E\simeq \ast \ar[d]^-{}\ar[r]&\simeq \ast\ar[d]\\
X\ar[r]^-{f}&\mm{BS\mathcal G}  \ar[r]^-{\gamma} &  \mm{B(\mathcal G/TOP)}\ar[r]&\mm{B(\mathcal G/TOP)}[2]
}
\]

The left fibration is the pull-back fibration $(\gamma\circ f)^\ast\mathcal U$ where $\mathcal U$ is the universal principle fibration $\mathcal U= \{\mathfrak E,p,\mm{B(\mathcal G/TOP)},\mm{\mathcal G/TOP}\}$.

From the last section, $\bar{\mathfrak E}$ is weakly homotopy equivalent to $\mm{BSTOP}$. Then we  regard the fibration $\mm{\mathcal G/TOP}\to \bar{\mathfrak E}\to \mm{BS\mathcal G}$ as 
$$\mm{\mathcal G/TOP}\stackrel{\mathfrak q}{\to} \mm{BSTOP}\to \mm{BS\mathcal G}$$

\begin{lemma}\label{varphi}
	If the homomorphism (\ref{gammaf}) is trivial for any $j\ge 1$, $\varphi^\ast:H^\ast(E;\mb{Z}/2)\to H^\ast(\mm{\mathcal G/TOP};\mb{Z}/2)$ is surjective.
\end{lemma}
\begin{proof}
By the $2$-localization, there is a map $\mm{\mathcal G/TOP}\to \mb{K}$ inducing an isomorphism $H_\ast(\mm{\mathcal G/TOP};\mb{Z}/2)\cong H_\ast(\mb{K};\mb{Z}/2)$.
 Let $\mb{K}=\mb{K}_1\times \mb{K}_2$ 
$$\mb{K}_1=\Pi_{j\ge 2} \mm{K}(\mb{Z}/2,2^j-2)$$
$$\mb{K}_2=\Pi_{k\ge 1} \mm{K}(\mb{Z}_{(2)},4k)\times \Pi_{k\ne 2^j}\mm{K}(\mb{Z}/2,4k-2)$$
Let ${\ell}_{2k}\in H_{2k}(\mm{\mathcal G/TOP};\mb{Z}/2)$ correspond to the fundamental class  
$$\text{$\iota_{2k}\in H_{2k}(\mm{K}(\mb{Z}/2,2k);\mb{Z}/2)$ or  $H_{2k}(\mm{K}(\mb{Z}_{(2)},2k);\mb{Z}/2)$}$$

By \cite{BMM1971,BMM1973}, it is injective that the restriction on $H_\ast(\mb{K}_2;\mb{Z}/2)$ of the homomorphism $\mathfrak q_\ast:H_\ast(\mm{\mathcal G/TOP};\mb{Z}/2)\to H_\ast(\mm{BSTOP};\mb{Z}/2)$. 
By the following commutative diagram 
\[
\xymatrix@C=0.8cm{
\mm{\mathcal G/TOP}\ar@{=}[r]\ar[d]^-{\varphi}&\mm{\mathcal G/TOP}\ar[d]^{q}\\
E \ar[r]&\mm{BSTOP}
}
\]
 it is also injective that the restriction on $H_\ast(\mb{K}_2;\mb{Z}/2)$ of the homomorphism $\varphi_\ast:H_\ast(\mm{\mathcal G/TOP};\mb{Z}/2)\to H_\ast(E;\mb{Z}/2)$.

 Next, we consider the Serre spectral sequences.
 
 The natural projection induces a morphism between the following spectral sequences
\begin{equation}
	E_2^{p,q}=H_p(\mm{B(\mathcal G/TOP)};H_q(\mm{\mathcal G/TOP};\mb{Z}/2))\Longrightarrow H_{p+q}(\ast;\mb{Z}/2) \label{unverfib}
\end{equation}
$$\bar{E}_2^{p,q}=H_p(\mm{K}(\mb{Z}/2,2^{j+1}-1) ;H_q(\mm{K}(\mb{Z}/2,2^{j+1}-2);\mb{Z}/2))\Longrightarrow H_{p+q}(\ast;\mb{Z}/2)$$
By the naturality, ${\ell}_{2^{j+1}-2}\in E_{2^{j+1}-1}^{0,2^{j+1}-2}$. 
 
Thus, in the Serre spectral sequence
\begin{equation}
\mb{E}_2^{p,q}=H_p(X;H_q(\mm{\mathcal G/TOP};\mb{Z}/2))\Longrightarrow H_{p+q}(E;\mb{Z}/2) \label{Xfib}	
\end{equation}
we first have that $${\ell}_{2^{j+1}-2}\in \mb{E}_{2^{j+1}-1}^{0,2^{j+1}-2}$$ by the naturality for the Serre spectral sequences. Furthermore, the nontrivial differential for ${\ell}_{2^{j+1}-2}$ may be
$$\mm{d}_{2^{j+1}-1}:\mb{E}_{2^{j+1}-1}^{2^{j+1}-1,0}\to \mb{E}_{2^{j+1}-1}^{0,2^{j+1}-2}.$$

Assume there exists an element $x\in \mb{E}_{2^{j+1}-1}^{2^{j+1}-1,0}$ such that 
$$\mm{d}_{2^{j+1}-1}(x)={\ell}_{2^{j+1}-2}$$ in the spectral sequence (\ref{Xfib}). By the naturality, 
\begin{equation}
	{d}_{2^{j+1}-1}((\gamma\circ f)_\ast x)={\ell}_{2^{j+1}-2} \in {E}_{2^{j+1}-1}^{0,2^{j+1}-2}\label{contradiction}
\end{equation}
in the spectral sequences (\ref{unverfib}).
Since the homomorphism (\ref{gammaf}) 
  is trivial,
   $(\gamma\circ f)_\ast:\mb{E}_{2^{j+1}-1}^{2^{j+1}-1,0}\to {E}_{2^{j+1}-1}^{2^{j+1}-1,0}$ is trivial which contradicts with the equation (\ref{contradiction}). Hence $${\ell}_{2^{j+1}-2} \in \mb{E}_{2^{j+1}}^{0,2^{j+1}-2}\cong \mb{E}_{\infty}^{0,2^{j+1}-2}$$ which means that $\varphi_\ast({\ell}_{2^{j+1}-2})\ne 0$ by Proposition \ref{edgehomo}.
   
   Let ${\ell}_{2k}^\ast\in H^{2k}(\mm{\mathcal G/TOP};\mb{Z}/2)$ be the image of the fundamental class  
$$\text{$\iota_{2k}^\ast\in H^{2k}(\mm{K}(\mb{Z}/2,2k);\mb{Z}/2)$ or  $H^{2k}(\mm{K}(\mb{Z}_{(2)},2k);\mb{Z}/2)$}$$
under the isomorphism $H^\ast(\mb{K};\mb{Z}/2)\to H^\ast(\mm{\mathcal G/TOP};\mb{Z}/2)$. Subsequently, ${\ell}_{2k}^\ast$ is a dual of ${\ell}_{2k}$. 
 
 Since $0\ne \varphi_\ast({\ell}_{2k})=e_{2k}\in H_{2k}(E;\mb{Z}/2)$, there exists a dual $e_{2k}^\ast\in H^{2k}(E;\mb{Z}/2)$ of $e_{2k}$ such that $$\varphi^\ast(e_{2k}^\ast)={\ell}_{2k}^\ast+\mathfrak X_{2k}\in H^\ast(\mm{\mathcal G/TOP};\mb{Z}/2)$$ where $\mathfrak X_{2k}\in H^{2k}(\Pi_{2j<k}\mm{K}(\mb{Z}_{(2)},4j)\times \Pi_{2j<k-1}\mm{K}(\mb{Z}/{2},4j+2);\mb{Z}/2)$.
In particular, $\varphi^\ast(e_{2}^\ast)={\ell}_{2}^\ast$.

Recall that $H^\ast(\mm{K}(\mathbb{Z}_{(2)},2k);\mathbb{Z}/2)$ and $H^\ast(\mm{K}(\mathbb{Z}/{2},2k);\mathbb{Z}/2)$ are polynomial rings with generators ${\ell}_{2k}$ and $\mathrm{Sq}^\mm{I} {\ell}_{2k}$, where $\mm{I}=(i_1,i_2,\cdots,i_r)$ is a admissible sequence \cite{Sw} satisfying
$i_1-i_2-\cdots -i_r < 2k$, $i_r>1$ for the case of $\mm{K}(\mathbb{Z}_{(2)},2k)$, $i_r>0$ for the case of $\mm{K}(\mathbb{Z}/{2},2k)$. 
Note $$\varphi^\ast:H^\ast(E;\mb{Z}/2)\to H^\ast(\mm{\mathcal G/TOP};\mb{Z}/2)$$
is an algebra homomorphism over the mod $2$ Steenrod algebra. By $\varphi^\ast(e_{2}^\ast)={\ell}_{2}^\ast$, we have $\varphi^\ast(\mm{Sq}^1e_{2}^\ast)=\mm{Sq}^1{\ell}_{2}^\ast$. Hence the image of $\varphi^\ast$ contains $H^\ast(\mm{K}(\mathbb{Z}/{2},2);\mathbb{Z}/2)$. Note that $\varphi^\ast(e_{4}^\ast)={\ell}_{4}^\ast+\mathfrak X_4$ and $\mathfrak X_4\in H^4(\mm{K}(\mathbb{Z}/{2},2);\mathbb{Z}/2)$. Thus, there exists $\bar{e}_{4}^\ast\in H^4(E;\mb{Z}/2)$ such that $\varphi^\ast(\bar{e}_{4}^\ast)={\ell}_{4}^\ast$. 

By induction, there exists $\bar{e}_{2k}^\ast\in H^{2k}(E;\mb{Z}/2)$ such that $\varphi^\ast(\bar{e}_{2k}^\ast)={\ell}_{2k}^\ast$ for $k\ge 1$. Then
the desired epimorphism follows from the algebra homomorphism $\varphi^\ast$ over the mod $2$ Steenrod algebra.     
\end{proof}

\begin{proposition}\label{EinftyE2}
	If the homomorphism (\ref{gammaf}) is trivial for any $j\ge 1$, $E_\infty^{p,q}\cong E_2^{p,q}$ in the Serre spectral sequence for the pull-back fibration
	$$\mm{\mathcal G/TOP}\stackrel{\varphi}{\to} E\to X$$
	$$E_2^{p,q}=H^p(X;H^q(\mm{\mathcal G/TOP};\mb{Z}/2))\Longrightarrow H^{p+q}(E;\mb{Z}/2)$$
\end{proposition}
\begin{proof}
	By Lemma \ref{varphi}, $\varphi^\ast:H^{\ast}(E;\mb{Z}/2)\to H^\ast(\mm{\mathcal G/TOP};\mb{Z}/2)$ is an epimorphism when the homomorphism (\ref{gammaf}) is trivial for any $j\ge 1$. By the dual version of Proposition \ref{edgehomo}, $E_\infty^{0,q}=E_2^{0,q}$ for $q\ge 0$.
	
	Note $E_2^{p,q}=E_2^{p,0}\otimes E_2^{0,q}$ and $d_2(E_2^{0,q})=0$. $d_2(E_2^{p,q})=0$ for $p,q\ge 0$ by the formula \cite{Sw} 
	$$d_r(x\otimes y)=d_r(x)\otimes y+(-1)^{\mm{deg}(x)}x\otimes d_r(y)$$  for $r\ge 2$. Thus $E_3^{p,q}=E_3^{p,0}\otimes E_3^{0,q}=E_2^{p,0}\otimes E_2^{0,q}$. 
	
	By induction, $E_\infty^{p,q}\cong E_2^{p,q}$ for $p,q\ge 0$.
\end{proof}

From the above proposition, we have

\begin{theorem}\label{mod2homo}
	If the homomorphism (\ref{gammaf}) is trivial for $j\ge 1$, then	$(\gamma\circ f)_\ast:H_{2k}(X;\mb{Z}/2)\to H_{2k}(\mm{B(\mathcal G/TOP)};\mb{Z}/2)$ is trivial for $k\ge 1$.
\end{theorem} 
\begin{proof}
	Assume $(\gamma\circ f)_\ast:H_{2k}(X;\mb{Z}/2)\to H_{2k}(\mm{B(\mathcal G/TOP)};\mb{Z}/2)$ is nontrivial for some $k$. Consider the dual of $(\gamma\circ f)_\ast$. There exists an element $b\in H^{2k}(\mm{B(\mathcal G/TOP)};\mb{Z}/2)$ such that $$(\gamma\circ f)^\ast(b)\ne 0\in H^{2k}(X;\mb{Z}/2).$$
	
	In the spectral sequence
$$	\mb{E}_2^{p,q}=H^p(\mm{B(\mathcal G/TOP)};H^q(\mm{\mathcal G/TOP};\mb{Z}/2))\Longrightarrow H^{p+q}(\ast;\mb{Z}/2)$$
for some $r$, there exists a nontrivial differential 
$$d_r:\mb{E}_r^{2k-r,r-1}\to \mb{E}_r^{2k,0}$$
such that $b\in \mb{E}_r^{2k,0}$ and $b\in \mm{im}(d_r)$, i.e., $b=0\in \mb{E}_{r+1}^{2k,0}$. 

By the naturality, $(\gamma\circ f)^\ast(b)=0\in {E}_{r+1}^{2k,0}$  in the spectral sequence
$${E}_2^{p,q}=H^p(X;H^q(\mm{\mathcal G/TOP};\mb{Z}/2))\Longrightarrow H^{p+q}(E;\mb{Z}/2)$$
which contradicts with ${E}_{\infty}^{2k,0}={E}_{2}^{2k,0}$ (Proposition \ref{EinftyE2}).
\end{proof}

\section{The $\mb{Q}$-homology homomorphism}\label{SecQhomo}
 In this section, $X$ is a finite 1-connected CW-complex and satisfies the torsion-free condition (\ref{torsionfree02}). Let $f:X\to \mm{BS}\mathcal G$ be a classifying map of a stable spherical fibration over $X$.
 
 Consider the homomorphism of $\mb{Q}$-coefficient homology groups
$$H_{4k+1}(X;\mb{Q})\stackrel{f_\ast}{\to }H_{4k+1}(\mm{BS}\mathcal G;\mb{Q})\stackrel{\gamma_\ast}{\to }H_{4k+1}(\mm{B(\mathcal G/TOP)};\mb{Q})$$
Recall that $\pi_{n}\mm{BS\mathcal G}$ is a finite  group for $n\ge 0$ (see Example \ref{ex}). By the exact sequence
$$ \pi_{n+1}(\mm{BS\mathcal G})\to \pi_{n}(\mm{\mathcal G/TOP})\stackrel{\mathfrak q_\ast}{\to} \pi_{n}(\mm{BSTOP})\to \pi_n(\mm{BS\mathcal G})$$
 $\mathfrak q: \mm{\mathcal G/TOP}\to \mm{BSTOP}$ 
 defines a rational equivalence \cite{FHT}, i.e. 
 $$\mathfrak q_\ast\otimes \mb{Q} :\pi_\ast(\mm{\mathcal G/TOP})\otimes\mb{Q}\to \pi_\ast(\mm{BSTOP})\otimes\mb{Q}$$ 
is isomorphic. Hence $\mathfrak q_\ast :H_\ast(\mm{\mathcal G/TOP};\mb{Q})\to H_\ast(\mm{BSTOP};\mb{Q})$ 
is isomorphic. Furthermore, $H_n(\mm{BS\mathcal G};\mb{Q})=0$ for any $n\ge 1$ by the $\mb{Q}$-coefficient version of Proposition \ref{edgehomo}. Thus we have
\begin{theorem}\label{Qcoefficienthomo}
	For any map $f:X\to \mm{BS}\mathcal G$, the homomorphism 
	$$H_{n}(X;\mb{Q})\stackrel{f_\ast}{\to }H_{n}(\mm{BS}\mathcal G;\mb{Q})\stackrel{\gamma_\ast}{\to }H_{n}(\mm{B(\mathcal G/TOP)};\mb{Q})$$ 
	is trivial for $n>0$.
\end{theorem}

Now we present the proof of Theorem \ref{2nfree}:
 By Theorem \ref{mod2homo} and \ref{Qcoefficienthomo}, the homomorphisms (\ref{4njia1}) and (\ref{4njia3}) are trivial if the homomorphism
 $$(\gamma\circ f)_\ast:H_{2^{j+1}-1}(X;\mb{Z}/2)\to H_{2^{j+1}-1}(\mm{B(\mathcal G/TOP)};\mb{Z}/2)$$
is trivial for $j>0$. From the argument at the end of Section \ref{Sechomoonhomology}, Theorem \ref{2nfree} follows.

\section{Homeomorphism types and homotopy types}\label{SecSurgery} 

From \cite{Sull1966}, we have a one-to-one correspondence 
$$\mm{hT}(M)\to [M-\{pt\},\mm{\mathcal G/PL}]$$
$M$ is a closed piecewise linear (PL) manifold \cite{RoSa}, $\mm{\mathcal G/PL}$ is the homotopy fiber of the natural map
$\mm{BPL}\to \mm{B\mathcal G}$ where $\mm{BPL}$ is the classifying space of stable piecewise linear bundles. Let $\mm{HT}(M)$ be the set of PL manifolds homotopy equivalent to $M$. $M_1\in \mm{HT}(M)$ is concordant to $M_2\in \mm{HT}(M)$ if $M_1$ is PL-homeomorphic to $M_2$. $\mm{hT}(M)$ is the set of concordance classes. 


Let $\mathcal M$ be an $(n-1)$-connected topological manifold of dimension $2n$, which indeed is a PL manifold (cf. \cite{KirbSieb1977}). Furthermore, any two such manifolds are PL-homeomorphic if and only if they are homeomorphic. From \cite{KirSie1969}, closed topological manifolds of dimension other than four are homeomorphic to CW complexes. Therefore, by the simplest cell structure \cite{Wall1965}, $\mathcal M-\{pt\}$ is homotopy equivalent to $\vee S^n$. 
$$[\mathcal M-\{pt\},\mm{\mathcal G/PL}]\cong [\vee S^n,\mm{\mathcal G/PL}]$$ 
By \cite{Sull1967}, $\pi_{2i+1}(\mm{\mathcal G/PL})=0$. Hence, in case of $n$ is odd, the number of the set $\mm{hT}(\mathcal M)$ equals $1$, which implies that a closed manifold of dimension $2n$ is homeomorphic to $\mathcal M$ if and only if it is homotopy equivalent to $\mathcal M$. Thus, from Theorem \ref{homotomanifold}, we have the corollary \ref{onetoone} as follows 

\begin{corollary}\label{onetooneinSec}
	Let $n$ be odd, and $n\ne 2^j-1$ for any $j\ge 1$.
The homotopy types of finite $(n-1)$-connected Poincar$\mm{\acute e}$ complexes of dimension $2n$ are in one-to-one correspondence with the homeomorphism types of $(n-1)$-connected topological manifolds of dimension $2n$. 
\end{corollary}

Let $P$ be such Poincar$\mm{\acute e}$ complex as in Corollary \ref{onetooneinSec}. By the simplest cell structure and the Poincar$\mm{\acute e}$ dual, its cell structure is 
$(\vee_{2k} S^n)\cup_{\beta} e^{2n} $
where $\vee_{2k} S^n$ is the wedge product of $2k$ copies of $S^n$, $\beta\in \pi_{2n-1}(\vee_{2k} S^n)$ is the attaching map. 
Next we focus on the attaching map. 

By the work of Hilton (see \cite{WhiteGW1978}),
	$$\pi_{2n-1}(\vee_{2k} S^n)=\oplus^{2k}  \pi_{2n-1}(S^n)\oplus^{C_{2k}^2} \mb{Z}$$ where $C_{2k}^2$ is the combination number. Let $S_i^n$ be the i-th $S^n\subset \vee_{2k} S^n$. Each $\mb{Z}$ of $\oplus^{C_{2k}^2} \mb{Z}$ is generated by the whitehead product
	$[\iota_i^n,\iota_j^n]$ where $1\le i<j\le 2k$, $\iota_i^n$ is  a generate of $\pi_n(S^n_i)$. 
\begin{proposition}\label{whitepro}
Let $n>1$ be odd. For a general CW-complex $X=(\vee_{2k} S^n)\cup_{f} e^{2n}$,	if the attaching map $f\in \pi_{2n-1}(\vee_{2k} S^n)$ has a component $s[\iota_i^n,\iota_j^n]$, $s\in \mb{Z}$, then for the generators $\ell_i\in H^n(S_i^n;\mb{Z})$, $L\in H^{2n}(X;\mb{Z})$, $\ell_i\cup \ell_j=\pm sL$. 
\end{proposition}
\begin{proof}
	Take the projection $\vee_{2k} S^n\to S_i^n\vee S_j^n$, which maps the element $f$ to $f_{ij}$. Indeed, $f_{ij}\in \pi_{2n-1}(S_i^n\vee S_j^n)$ has a component $s[\iota_i^n,\iota_j^n]$. Furthermore, the above projection has an extension map	\begin{equation}
		g:X\to X_{ij}=(S_i^n\vee S_j^n)\cup_{f_{ij}}e^{2n} \label{gmap}
	\end{equation}
	so that $g^\ast:H^{2n}(X_{ij};\mb{Z})\to H^{2n}(X;\mb{Z})$ is an isomorphism. 
	
	Note that $n$ is odd. By the work of Serre \cite{Toda}, $\pi_{2n-1}(S^n)$ is a finite abelian group. Let $f_{ij}=(s[\iota_i^n,\iota_j^n],\mm{T}_1,\mm{T}_2)\in \mb{Z}\oplus \pi_{2n-1}(S^n)\oplus \pi_{2n-1}(S^n)$, the order of $\mm{T}_1$ be $m_1$, the order of $\mm{T}_2$ be $m_2$, $m=m_1\cdot m_2$. 
  Then $$ms[\iota_i^n,\iota_j^n]=mf_{ij}=0\in \pi_{2n-1}(X_{ij}).$$ 
	Recall [\cite{WhiteGW1978}, Theorem 7.7] $S^n\times S^n=(S_i^n\vee S_j^n)\cup_{[\iota_i^n,\iota_j^n]} e^{2n} $. Note $ms[\iota_i^n,\iota_j^n]=[ms\iota_i^n,\iota_j^n]$ \cite{WhiteheadJHC}. For the map
	$S_i^n\vee S_j^n\stackrel{(\times ms)\vee id}{\longrightarrow }S_i^n\vee S_j^n$,
	we have an extension map
	$$F:S^n\times S^n=(S_i^n\vee S_j^n)\cup_{[\iota_i^n,\iota_j^n]} e^{2n}\to X_{ij}=(S_i^n\vee S_j^n)\cup_{f_{ij}}e^{2n}$$ 
	 Moreover,  
	\begin{equation}
		F^\ast(\ell_i)=ms\bar{\ell}_i,\quad F^\ast(\ell_j)=\bar{\ell}_j \label{ndime}
	\end{equation}
	 where $\ell$ is a generator of $H^n(X_{ij};\mb{Z})$, $\bar{\ell}$  of $H^n(S^n\times S^n;\mb{Z})$. Consider the following commutative diagram
	\[
\xymatrix{
  \pi_{2n}(S^n\times S^n,S^n\vee S^n)=\mb{Z} \ar[r]^-{\bar{\partial}} \ar[d]^-{F} &\pi_{2n-1}(S^n\vee S^n) \ar[r]^-{} \ar[d]^-{(\times ms)\vee id}&\pi_{2n-1}(S^n\times S^n)  \ar[d]^-{{F}}\\
\pi_{2n}(X_{ij},S^n\vee S^n)=\mb{Z}\ar[r]^-{\partial} & \pi_{2n-1}(S^n\vee S^n)\ar[r]& \pi_{2n-1}(X_{ij})
}
\] 
$$((\times ms)\vee id)_\ast \circ \bar{\partial}(\bar{e})=(0,0,ms[\iota_i^n,\iota_j^n])\in \pi_{2n-1}(S^n)\oplus \pi_{2n-1}(S^n)\oplus \mb{Z}$$
$$\partial(e)=(\mm{T}_1,\mm{T}_2,s[\iota_i^n,\iota_j^n])\in \pi_{2n-1}(S^n)\oplus \pi_{2n-1}(S^n)\oplus \mb{Z}$$
where $\bar{e}$ is a generator of $\pi_{2n}(S^n\times S^n,S^n\vee S^n)$, ${e}$ of $\pi_{2n}(X_{ij},S^n\vee S^n)$. Hence $F_\ast(\bar e)=me$. By the Hurewicz theorem,
\[
\xymatrix{
 H_{2n}(S^n\times S^n;\mb{Z})\ar[r]^-{\cong}\ar[d]^-{F} & H_{2n}(S^n\times S^n,S^n\vee S^n;\mb{Z})  \ar[d]^-{F} \\
H_{2n}(X_{ij};\mb{Z})\ar[r]^-{\cong}& H_{2n}(X_{ij},S^n\vee S^n;\mb{Z})
}
\]
\begin{equation}
	F^\ast(L)=m\bar L\in H^{2n}(S^n\times S^n;\mb{Z}) \label{2ndime}
\end{equation}
where $L$ is a generator of $H^{2n}(X_{ij};\mb{Z})$, $\bar L$ of $H^{2n}(S^n\times S^n;\mb{Z})$. 

By equations (\ref{ndime}) and (\ref{2ndime}), $\ell_i\cup \ell_j=sL$. Then the desired result follows by the map $g$ (\ref{gmap}).
\end{proof}

Let $n>1$ be odd. For Poincar$\mm{\acute e}$ complex
$P=(\vee_{2k} S^n)\cup_{\beta} e^{2n}$, $H^n(P;\mb{Z})$ has a symplectic basis $\ell_{a_1},\ell_{a_2},\cdots,\ell_{a_k},\ell_{b_1},\cdots, \ell_{b_k}$ such that 
$$\ell_{a_i}\cup \ell_{b_j}=\ell_{a_i}\cup \ell_{a_h}=0,\quad \ell_{a_i}\cup \ell_{b_i}=L$$
where $1\le i,j,h\le k$, $i\ne j$, $L$ is a generator of $H^{2n}(P;\mb{Z})$. By Proposition \ref{whitepro}, $\beta$ has the following fixed component in free part
$$([\iota_{a_1},\iota_{b_1}],[\iota_{a_2},\iota_{b_2}],\cdots,[\iota_{a_k},\iota_{b_k}],0,\cdots,0)\in \oplus^{C_{2k}^2} \mb{Z} $$ 
where 
$\iota_{a_i}$, $\iota_{b_i}$ are the generators of $\pi_n(S^n_{a_i})$, $\pi_n(S^n_{b_i})$. Hence the homotopy type of $P$ is determined by the component of $\beta$ in $\oplus^{2k}  \pi_{2n-1}(S^n)$. 

Indeed, for Poincar$\mm{\acute e}$ complexes
$$P_1=(\vee_{2k} S^n)\cup_{\beta_1} e^{2n}\quad P_2=(\vee_{2k} S^n)\cup_{\beta_2} e^{2n}$$ even though the component of $\beta_1$ is different from the component of $\beta_2$ in $\oplus^{2k}  \pi_{2n-1}(S^n)$, $P_1$ may be also homotopy equivalent to $P_2$. This reason arises from the self homotopy equivalences of $\vee_{2k}S^n$. We will present some examples in the next section.  

\section{Certain highly connected complexes}\label{classifsec}

Let $n>1$ be odd. Recall the EHP sequence \cite{Jame} 
$$\pi_{2n}(S^n)\stackrel{\Sigma}{\to} \pi_{2n+1}(S^{n+1})\stackrel{H}{\to}\pi_{2n+1}(S^{2n+1})\to \pi_{2n-1}(S^{n})\to \pi_{n-1}^s\to 0$$
where $H$ is defined by the Hopf invariant, $\Sigma$ is the suspension. By the work of Serre, $\pi_i(S^m)$ is finite except for the cases that $i = m$,
and that $m$ even and $i = 2m-1$ (also see \cite{Toda}). Hence $\pi_{2n+1}(S^{n+1})$ contains only one $\mb{Z}$ direct summand that detected by the Hopf invariant. It is well known that there exists $\alpha\in \pi_{2n+1}(S^{n+1})$ so that its Hopf invariant equals to $2$ (see \cite{Hatcher}). 

By \cite{Adams1960}, for odd $n\ne 1,3,7$, there is no element of Hopf invariant one in $\pi_{2n+1}(S^{n+1})$. Hence we have the following exact sequence
$$0\to \mb{Z}/2\to \pi_{2n-1}(S^{n})\to \pi_{n-1}^s\to 0$$

By the computation in \cite{Toda},
in case of $\pi_{n-1}^s=0$, $n\ne 1,3,7$.
Next, we always assume that $n$ is odd such that $\pi_{n-1}^s=0$, i.e. $\pi_{2n-1}(S^{n})=\mb{Z}/2$. 

Now, we define the Kervaire invariant \cite{Kerva1960} for the Poincar$\mm{\acute e}$ complex
$$P=(\vee_{2k} S^n)\cup_{\beta} e^{2n}$$

Let $\Omega=\Omega S^{n+1}$ be the loop-space on the $(n+1)$-sphere. It is well known that
$H^n(\Omega;\mb{Z})=\mb{Z}\langle e_n\rangle $, $H^{2n}(\Omega;\mb{Z})=\mb{Z}\langle e_{2n}\rangle$. The $e_n\in H^n(\Omega;\mb{Z})$ can be represented as the map $\Omega\to \mm{K}(\mb{Z},n)$. Let its homotopy fiber be $F$. By the assumption for $n$, $\pi_{2n}(S^{n+1})=\pi_{n-1}^s=0$.  
\begin{equation}
	\pi_{2n-1}(F)=\pi_{2n-1}(\Omega S^{n+1})=\pi_{2n}(S^{n+1})=0 \label{piF}
\end{equation}
 For any $\ell\in H^n(P;\mb{Z})$, there is a map $f:P\to \Omega$ such that 
 $f^\ast(e_n)=\ell$ by obstruction theory \cite{Hatcher}.
  
 Define a function $\varphi_0: H^n(P;\mb{Z}) \to  \mb{Z}/2$ by the following device. For
 $\ell\in H^n(P;\mb{Z})$, take a map $f: P\to \Omega $ such that $f^\ast(e_n)=\ell$. Then, $\varphi_0(\ell)
= f^\ast(u_{2n}) [P]$ where $u_{2n}\in  H^{2n}(\Omega ; \mb{Z}/2)$ is the reduction modulo $2$ of $e_{2n} \in  H^{2n}(\Omega;\mb{Z})$, and $f^\ast(u_{2n}) [P]_2$ is the value of the cohomology class $f^\ast(u_{2n})$ on the reduction modulo $2$ of the fundamental class of $P$. 

Following \cite{Kerva1960}, the function $\varphi_0: H^n(P;\mb{Z}) \to  \mb{Z}/2$ is well defined. Let $l\in H^n(P;\mb{Z}/2)$ be the reduction modulo $2$ of $\ell$.
 We define $\varphi: H^n(P;\mb{Z}/2) \to  \mb{Z}/2$ by $\varphi(l)=\varphi_0(\ell)$ which follows from $\varphi_0(2\ell)=0 .$ Let $l_{a_1},l_{a_2},\cdots,l_{a_k},l_{b_1},\cdots, l_{b_k}$ be a symplectic basis of $H^n(P;\mb{Z}/2)$, 
$$\Phi(P)=\Sigma_1^k\varphi(l_{a_i})\cdot \varphi(l_{b_i}) $$ 
Indeed, the Kervaire invariant $\Phi(P)$ is a homotopy invariant, and 
independent of the symplectic basis $l_{a_1},\cdots,l_{a_k},l_{b_1},\cdots, l_{b_k}$. 

By the simplest cell structure, we have the cell structure
\begin{equation}
	\Omega=S^n\cup_\alpha e^{2n}\cup \cdots \label{cellOmega}
\end{equation}
where $\alpha$ is the attaching map. Since $\pi_{2n-1}(S^n)=\mb{Z}/2$ and the equation (\ref{piF}), 
 $\alpha$ is the generator of $\pi_{2n-1}(S^n)$.

Let $P$ be the Poincar$\mm{\acute e}$ complex of dimension $2n$. We take a symplectic basis $\ell_{1},\ell_2,\cdots,\ell_{2k-1}, \ell_{2k}\in H^{n}(P;\mb{Z})$ so that $\ell_{2i-1}\cup \ell_{2i}$ is a generator of $H^{2n}(P;\mb{Z})$ for $1\le i\le k$.
	$$ P=(S_1^n\vee S_2^n\vee \cdots \vee S_{2k-1}^n\vee S_{2k}^n)\cup_{\beta} e^{2n}$$
	Let $\ell_{j}$ be a generator of $H^n(S_j^n;\mb{Z})$. Note that 
	$$\pi_{2n-1}(S_1^n\vee S_2^n\vee \cdots \vee S_{2k-1}^n\vee S_{2k}^n)=\oplus_{i=1}^{2k} \pi_{2n-1}(S_i^n)\oplus^{C_{2k}^2}\mb{Z}$$
	
By Proposition \ref{whitepro}, the component of ${\beta}$ in free part is 
$$[\iota_{1},\iota_{2}]+[\iota_{3},\iota_{4}]+\cdots+[\iota_{{2k-1}},\iota_{{2k}}]$$
where $\iota_j$ is a generator of $\pi_n(S^n_j)=\mb{Z}$.

Let $\beta_j$ be the component of $\beta$ in $\pi_{2n-1}(S_j^n)=\mb{Z}/2$, the array of $\beta$ be 
$$(\beta_1,\beta_2,\cdots ,\beta_{2k-1},\beta_{2k})\in \oplus_{j=1}^{2k} \pi_{2n-1}(S_j^n)=(\mb{Z}/2)^{2k}$$
Here we can also regard $\beta_j$ as the number $0$ or $1$.
\begin{lemma}\label{kevarinv}
For $\varphi_0:H^n(P;\mb{Z})\to \mb{Z}/2$, $\varphi_0(\ell_j)=\beta_j$.	
\end{lemma}
\begin{proof}
Considering the projection 
$$\mathfrak p:P\to P_j=S^n_j\cup_{\beta_j}e^{2n} $$
we have that $\mathfrak p^\ast:H^i(P_j;\mb{Z})\to H^i(P;\mb{Z})$ is a monomorphism for $i=n$, an isomorphism for $i=2n$. Let $\hat{\ell}_j$ denote a generator of $H^n(P_j;\mb{Z})$ such that $\mathfrak p^\ast(\hat{\ell}_j)={\ell}_j$. By equation (\ref{piF}), the map $\hat{\ell}_j:P_j\to \mm{K}(\mb{Z},n)$ has a lift $\hat{\ell}_j^\prime:P_j\to \Omega$. Similarly, we also have a well defined function $\hat{\varphi}_0:H^n(P_j;\mb{Z})\to \mb{Z}_2$. Moreover,
$\varphi_0(\ell_j)=\hat{\varphi}_0(\hat{\ell}_j )$ since that the definition of $\varphi_0$ is independent of the choice of the lift $P\to \Omega$ of the map $\ell_j:P\to \mm{K}(\mb{Z},n)$ (cf. Lemma 1.2 in \cite{Kerva1960}).  
  
If $\beta_j=1$, then $P_j$ is homotopy equivalent to the $2n$-skeleton of $\Omega$. So the lift $\hat{\ell}_j^\prime:P_j\to \Omega$ is the inclusion. By the definition of $\hat{\varphi}_0$, we have $\hat{\varphi}_0(\hat{\ell_j})=1$.

If $\beta_j=0$, then $P_j=S^n_j\vee S^{2n}$. Applying the Hurewicz homomorphism, we have the commutative diagram
\[
\xymatrix@C=0.8cm{
\pi_{2n}(S^n_j\vee S^{2n})\ar[d]^-{}\ar[r]^-{\hat{\ell}_j^\prime}&\pi_{2n}(\Omega)=\mb{Z}\oplus \mm{im}\Sigma \ar[d]^-{}\\
H_{2n}(S^n_j\vee S^{2n};\mb{Z})\ar[r]^-{\hat{\ell}_j^\prime}&  H_{2n}(\Omega;\mb{Z})=\mb{Z}
}
\]
where $\Sigma:\pi_{2n}(S^n)\to \pi_{2n+1}(S^{n+1})\cong \pi_{2n}(\Omega)$ is the suspension.  
It is well known that the right-hand Hurewicz homomorphism coincides with the Hopf invariant. By the assumption for $n$, $\mb{Z}\subset \pi_{2n}(\Omega)\to H_{2n}(\Omega;\mb{Z})$ is the $\times 2$ homomorphism (cf. \cite{Adams1960}). Therefore, by the above commutative diagram, we have that $(\hat{\ell}_j^\prime)_\ast:H_{2n}(P_j;\mb{Z}/2)\to H_{2n}(\Omega;\mb{Z}/2)$ is trivial, i.e. $(\hat{\ell}_j^\prime)^\ast:H^{2n}(\Omega;\mb{Z}/2)\to H^{2n}(P_j;\mb{Z}/2)$ is trivial. Then $\hat{\varphi}_0(\hat{\ell}_j)=0$ follows 
by the definition of $\hat{\varphi}_0$. Thus the desired result follows.
\end{proof}
\begin{theorem}\label{kervaintheor}
	The homotopy types of $(n-1)$-connected Poincar$\mm{\acute e}$ complexes of dimension $2n$ are determined by the $n$-th Betti number and the Kervaire invariant where $n$ is odd such that $\pi_{n-1}^s=0$.  
\end{theorem}
\begin{proof}
Let $P$ be the Poincar$\mm{\acute e}$ complex with $H^n(P)=\mb{Z}^{2k}$. Take $\ell_i$ $1\le i\le 2k$ as a symplectic basis of $H^n(P;\mb{Z})$. Then, by the simplest cell structure, we have  	
$$ P=(S_1^n\vee S_2^n\vee \cdots \vee S_{2k-1}^n\vee S_{2k}^n)\cup_{\beta} e^{2n}$$
where $\ell_i$ is a generator of $H^n(S^n_i)$.
Moreover, the component of $\beta$ in the free part is 
$[\iota_1,\iota_2]+[\iota_3,\iota_4]+\cdots +[\iota_{2k-1},\iota_{2k}]$
where $\iota_i\in \pi_{2n-1}(S^n_i)$. Denote the array of $\beta$ by
$(\beta_1,\beta_2,\cdots,\beta_{2k-1},\beta_{2k})$. By Lemma \ref{kevarinv},
\begin{equation}
	\Phi(P)=\Sigma_{j=1}^k \varphi_0(\ell_{2j-1})\varphi_0(\ell_{2j})=\Sigma_{j=1}^k \beta_{2j-1}\beta_{2j} \label{PhiP}
\end{equation} 
where $\beta_{2j-1}\beta_{2j}$ is the product in $\mb{Z}/2$.

 For a pair $(\beta_{2j-1},\beta_{2j})$ $1\le j\le k$, if it equals to $(0,1)$, we take  
$$\bar{\mathfrak h}: S^n_{2j-1}\vee S^n_{2j} \stackrel{\Delta \vee 1}{\longrightarrow} S^n_{2j-1}\vee S^n_{2j-1}\vee S^n_{2j}\stackrel{1\vee r\vee 1}{\longrightarrow }S^n_{2j-1}\vee S^n_{2j}$$
where $\Delta:S^n\to S^n\vee S^n$ is the coproduct of $S^n$, $r:S^n_{2j-1}\to S^n_{2j}$ is the map with degree $1$. It is easy to check that the homomorphism $\bar{\mathfrak h}_\ast:H_n(S^n_{2j-1}\vee S^n_{2j})\to H_n(S^n_{2j-1}\vee S^n_{2j})$ is isomorphic. Hence $\bar{\mathfrak h}$ is a homotopy equivalence. Moreover, we have 
$$\bar{\mathfrak h}_\ast[\iota_{2j-1},\iota_{2j}]=[\iota_{2j-1}+\iota_{2j},\iota_{2j}]=[\iota_{2j-1},\iota_{2j}]+[\iota_{2j},\iota_{2j}]$$
By \cite{Adams1960}, $0\ne [\iota_{2j},\iota_{2j}]$ is the generator of $\pi_{2n-1}(S^n_{2j})=\mb{Z}/2$ (here note that n is odd so that $\pi_{n-1}^s=0$). Then taking the homotopy equivalence
$$\mathfrak h=1\vee 1\vee \cdots \vee \bar{\mathfrak h}\vee \cdots \vee 1\vee 1:\vee_{2k}S^n\to \vee_{2k}S^n$$
we have that the component of $\hat{\beta}=\mathfrak h_\ast(\beta)$ in free part is also 
$$[\iota_1,\iota_2]+[\iota_3,\iota_4]+\cdots +[\iota_{2k-1},\iota_{2k}],$$
and the array of $\hat{\beta}$ is 
$({\beta}_1,{\beta}_2,\cdots,{\beta}_{2j-1},\hat{\beta}_{2j},\cdots,{\beta}_{2k-1},{\beta}_{2k})$. Since that ${\beta}_{2j-1}=0$, $\hat{\beta}_{2j}=\beta_{2j}+[\iota_{2j},\iota_{2j}]=0$. For $(\beta_{2j-1},\beta_{2j})=(1,0)$, we can also construct a similar map. 

Let $\#(\beta)$ be the number of the elements of the set 
$$\{(\beta_{2j-1},\beta_{2j})|\beta_{2j-1}=\beta_{2j}=1,1\le j\le k\}.$$

Case 1: $\Phi(P)=0$. By equation (\ref{PhiP}), $\#(\beta)$ is even. Assume
$$\beta_1=\beta_2=\beta_3=\beta_4=1$$
We define a map $\bar{\mathfrak h}$ by the composite of the following maps
$$ S^n_{1}\vee S^n_{2}\vee S^n_{3}\vee S^n_{4} \stackrel{\Delta \vee 1\Delta \vee 1}{\longrightarrow} S^n_{1}\vee S^n_{1}\vee S^n_{2}\vee S^n_{3}\vee S^n_{3}\vee S^n_{4}$$
$$S^n_{1}\vee S^n_{1}\vee S^n_{2}\vee S^n_{3}\vee S^n_{3}\vee S^n_{4}\stackrel{1\vee r^1\vee 1\vee 1\vee r^2\vee 1}{\longrightarrow }S^n_{1}\vee S^n_{2}\vee S^n_{3}\vee S^n_{4}$$
where $r^1:S^n_1\to S^n_4$ and $r^2:S^n_3\to S^n_2$ are the maps with degree $1$. Note that $\bar{\mathfrak h}$ is a homotopy equivalence.
 Since $[\iota_4,\iota_2]=-[\iota_2,\iota_4]$, 
$$\bar{\mathfrak h}_\ast([\iota_{1},\iota_{2}]+[\iota_{3},\iota_{4}])=[\iota_{1}+\iota_{4},\iota_{2}]+[\iota_3+\iota_2,\iota_4]=[\iota_{1},\iota_{2}]+[\iota_{3},\iota_{4}]$$
Then taking the homotopy equivalence
$$\mathfrak h=\bar{\mathfrak h}\vee 1\vee 1 \vee \cdots  \vee 1\vee 1:\vee_{2k}S^n\to \vee_{2k}S^n$$
we have that the component of $\hat{\beta}=\mathfrak h_\ast(\beta)$ in free part is also 
$$[\iota_1,\iota_2]+[\iota_3,\iota_4]+\cdots +[\iota_{2k-1},\iota_{2k}],$$
but the array of $\hat{\beta}$ is 
$({\beta}_1,\hat{\beta}_2, {\beta}_3,\hat{\beta}_4,{\beta}_{5},{\beta}_{6},\cdots,{\beta}_{2k-1},{\beta}_{2k})$
where $\hat{\beta}_{2}={\beta}_{2}+{\beta}_{3}=0$ and $\hat{\beta}_{4}={\beta}_{1}+{\beta}_{4}=0$. Applying the map constructed for pair $(1,0)$, we get a new map $\bar\beta$ of which the component equals to the component of ${\beta}$ in free part, the array is 
$(0,0,0,0,\beta_5,\beta_6, \cdots,{\beta}_{2k-1},{\beta}_{2k}).$

Repeat the above process. There exists a self homotopy equivalence $\mathfrak h:\vee_{i=1}^{2k} S^n_i\to \vee_{i=1}^{2k} S^n_i$ such that $\mathfrak h_\ast(\beta)=\bar{\beta}$ of which the component equals to the component of ${\beta}$ in free part, the array is $(0,0,\cdots,0,0)$. Moreover, the map $\mathfrak h$ can be extended to the Poincar$\mm{\acute e}$ complexes
$$P\to \bar P=\vee_{i=1}^{2k} S^n_i\cup_{\bar \beta}e^{2n}$$
 which is also a homotopy equivalence.
 
 Case 2: $\Phi(P)=1$. Note that $\#(\beta)$ is odd. We use the same method to construct the homotopy equivalence
 $$P\to \tilde P=\vee_{i=1}^{2k} S^n_i\cup_{\tilde \beta}e^{2n}$$
where the component of $\tilde \beta$ equals to the component of ${\beta}$ in free part, the array of $\tilde \beta$ is $(\tilde \beta_1,\tilde \beta_2,\cdots,\tilde \beta_{2k-1},\tilde \beta_{2k})$ so that there exists only one pair $(\tilde \beta_{2j-1},\tilde \beta_{2j})$ $1\le j\le k$ equaling to $(1,1)$. In particular, we let $j=1$ by a permutation.
\end{proof}

From the proof of Theorem \ref{kervaintheor}, the Poincar$\mm{\acute e}$ complex of dimension $2n$ of Kervaire invariant one always exists for any odd number $n$ satisfying $\pi_{n-1}^s=0$ and any value of the $n$-th Betti number.

By Corollary \ref{onetooneinSec} and Theorem \ref{kervaintheor}, the homeomorphism types of $(n-1)$-connected topological manifolds of dimension $2n$ are determined by the $n$-th Betti number and the Kervaire invariant if $n\ne 2^j-1$ is odd so that $\pi_{n-1}^s=0$.

\section{Detection of smooth structure}\label{smoothsec}

By Corollary \ref{onetooneinSec} and Theorem \ref{kervaintheor}, an $(n-1)$-connected topological manifolds of dimension $2n$ of Kervaire invariant zero is homeomorphic to the connected sum $\#^{k} S^n\times S^n $ where $2k$ is the $n$-th Betti number, $n\ne 2^j-1$ is odd so that $\pi_{n-1}^s=0$. 
In this section, we investigate the smooth structure of the $(n-1)$-connected topological manifold of dimension $2n$ of Kervaire invariant one.

\begin{definition}
	A connected smooth manifold $M$ is almost parallelizable if it is trivial that the restriction bundle of the stable normal bundle of $M$ on $M - x_0$
   for a point $x_0\in M$. 
\end{definition}
 
For an $(n-1)$-connected smooth manifold $M$ of dimension $2n$, $M-x_0$ is homotopy equivalent to the $n$ skeleton $\vee_{i} S_i^n$ of $M$. If $M$ is almost parallelizable,
 there is only one possibly non-vanishing obstruction $\mathfrak o_n\in H^{2n}(M;\pi_{2n-1}(\mm{SO}))$ to the construction of the trivialization for the stable normal bundle of $M$. $\mathfrak o_n$ is in the kernel of the Hopf-Whitehead homomorphism $J:\pi_{2n-1}(\mm{SO})\to \pi_{2n-1}^s$ by Lemma 1 of \cite{MilnorK1858}. Let $n$ be odd. Thus $2n=2$ or $6$ mod $8$. Note that $\pi_{2n-1}(\mm{SO})=\pi_{2n}(\mm{BSO})=0$ for $2n=6$ mod $8$. On the other hand, by \cite{Adams1966}, $\mm{ker}J_{2n-1}=0$ for $2n=2$ mod $8$. Hence, we have
\begin{lemma}\label{framed}
	Let $n$ be odd. 
	If an $(n-1)$-connected smooth manifold $M$ of dimension $2n$ is almost parallelizable, it is a framed manifold.
	\end{lemma} 

Finally, we give the main result of this section.  
\begin{corollary}
	Let $n$ is odd such that $\pi_{n-1}^s=0$. There exists an $(n-1)$-connected topological manifold of dimension $2n$ whose Kervaire invariant is one. Furthermore, it does not admit any smooth structure when additionally $n\ne 1\mod 8$ and $n\ne 2^j-1$.   
\end{corollary}
\begin{proof}
The existence of such manifold follows from Corollary \ref{onetooneinSec}, and the existence of $(n-1)$-connected $2n$-dimensional Poincar$\mm{\acute e}$ complex of Kervaire invariant one (cf. the proof of Theorem \ref{kervaintheor}).

By the Bott periodicity, an $(n-1)$-connected smooth manifold $M$ of dimension $2n$ is almost parallelizable when $n\ne 1\mod 8$ is odd. Hence, by Lemma \ref{framed}, such smooth manifold $M$ is framed. On the other hand, the Kervaire invariant of a framed manifold is zero for dimensions $\ne 2^{j+1}-2$ (cf. \cite{BrowderW1969}). Therefore, under the assumption for $n$, suppose that an $(n-1)$-connected $2n$-dimensional topological manifold of Kervaire invariant one admits a smooth structure, then a contradiction follows. 
\end{proof}

\begin{Ack}
I would like to thank Yang Su for many helpful discussions.	
\end{Ack}



\begin{thebibliography}{100}

\bibitem{Adams1960} {Adams}, J. F. On the non-existence of element of Hopf invariant one. Ann. of Math. (2) 72 (1960), 20-104.

\bibitem{Adams1966} {Adams}, J. F. On the groups J(x), IV. Topology 5 (1966), 21-71.


\bibitem{BoVo} Boardman, J. M.; Vogt, R. M.
Homotopy invariant algebraic structures on topological spaces. Lecture Notes in Math., Vol. 347. Springer-Verlag, Berlin-New York, 1973.

\bibitem{BrowderW1969} Browder, W.
The Kervaire invariant of framed manifolds and its generalization.
Ann. of Math. (2) 90 (1969), 157-186.

\bibitem{Brow1971} Browder, W.
Manifolds and homotopy theory. Manifolds-Amsterdam 1970 (Proc. Nuffic Summer School), pp. 17-35.

\bibitem{Brow1972} Browder, W. Poincar$\mm{\acute{e}}$ spaces, their normal fibrations and surgery. Invent. Math. 17 (1972), 191-202.

\bibitem{Brow1972book} Browder, W.
Surgery on simply-connected manifolds.
Ergeb. Math. Grenzgeb., Band 65.
Springer-Verlag, New York-Heidelberg, 1972.


\bibitem{BMM1971} Brumfiel, G.; Madsen, I.; Milgram, R. J.
PL characteristic classes and cobordism.
Bull. Amer. Math. Soc. 77 (1971), 1025-1030.

\bibitem{BMM1973} Brumfiel, G.; Madsen, I.; Milgram, R. J.
PL characteristic classes and cobordism.
Ann. of Math. (2) 97 (1973), 82-159.

\bibitem{DL} Dold, A.; Lashof, R.
Principal quasifibrations and fibre homotopy equivalence of bundles.
Illinois J. Math. 3 (1959), 285-305.

\bibitem{FHT} F$\mm{\acute{e}}$lix, Y.; Halperin, S.; Thomas, Jean-Claude. Rational homotopy theory. Grad. Texts in Math., 205. Springer-Verlag, New York, 2001.

\bibitem{Hatcher} Hatcher, A. Algebraic topology. Cambridge University Press, Cambridge, 2002.

\bibitem{Hun} Hungerford, T. W. Algebra.
Reprint of the 1974 original. 
Grad. Texts in Math., 73. Springer-Verlag, New York-Berlin, 1980.

\bibitem{Hu1994} Husemoller, D. Fibre bundles. Third edition. Grad. Texts in Math., 20
Springer-Verlag, New York, 1994.

\bibitem{Jame} James, I. M. On the Suspension Sequence. Ann. of Math. (2) 65 (1957), 74-107.

\bibitem{Kerva1960} Kervaire, M. A.
A manifold which does not admit any differentiable structure.
Comment. Math. Helv. 34 (1960), 257-270.

\bibitem{KirSie1969} Kirby, R. C.; Siebenmann, L. C.
On the triangulation of manifolds and the Hauptvermutung.
Bull. Amer. Math. Soc. 75 (1969), 742-749.

\bibitem{KirbSieb1977} {Kirby}, R. C.; Siebenmann, L. C. Foundational essays on topological Manifolds, smoothings and triangulations. Ann. of Math. Studies 88, Princeton Univ. Press, Princeton, NJ, 1977.

\bibitem{MaMi} Madsen, I.; Milgram, R. J.
The classifying spaces for surgery and cobordism of manifolds.
Ann. of Math. Stud., No. 92. Princeton University Press, Princeton, NJ; University of Tokyo Press, Tokyo, 1979.

\bibitem{Math1976} Mather, M.
Pull-backs in homotopy theory. Canadian J. Math. 28 (1976), no.2, 225-263.

\bibitem{MayPon2012} May, J. P.; Ponto, K.
More concise algebraic topology.
Localization, completion, and model categories.
Chicago Lectures in Math.
University of Chicago Press, Chicago, IL, 2012.

\bibitem{Mil19561} Milnor, J. W.
Construction of universal bundles. I.
Ann. of Math. (2) 63 (1956), 272-284.

\bibitem{Mil19562} Milnor, J. W.
Construction of universal bundles. II.
Ann. of Math. (2) 63 (1956), 430-436.

\bibitem{MilnorK1858} Milnor, J. W.; Kervaire, M. A.
Bernoulli numbers, homotopy groups, and a theorem of Rohlin. Proc. Internat. Congress Math. 1958.

\bibitem{Novi} Novikov, S. P. Homotopy equivalent smooth manifolds I. AMS Translations 48 (1965), 271-396.

\bibitem{RoSa} Rourke, C. P.; Sanderson, B. J. Introduction to piecewise-linear topology. Ergeb. Math. Grenzgeb., Band 69 [Results in Mathematics and Related Areas]
Springer-Verlag, New York-Heidelberg, 1972.

\bibitem{Rudyak} { Rudyak}, Y. B. On Thom spectra, orientability, and cobordism, Springer Mongraphs in Mathematics,Corrected 2nd printing, Springer, 2008.

\bibitem{Spi1967} Spivak, M. Spaces satisfying Poincar$\mm{\acute{e}}$ duality. Topology 6 (1967), 77-101.

\bibitem{Steen1951} Steenrod, N. E. The topology of fibre bundles. Princeton Mathematical Series 14, Princeton,
Princeton Univ. Press, Princeton, New Jersey 1951.

\bibitem{Sull1966} Sullivan, D.
Triangulating homotopy equivalences.
Thesis (Ph.D.), Princeton University, 1966.

\bibitem{Sull1967} Sullivan, D.
On the hauptvermutung for manifolds.
Bull. Amer. Math. Soc. 73 (1967), 598-600.

\bibitem{Sull2005} Sullivan, D.
Geometric topology: localization, periodicity and Galois symmetry. $K$-Monogr. Math., 8.  Springer, Dordrecht, 2005.

\bibitem{Sw} {Switzer}, R. W.
 Algebraic topology-homotopy and homology. Springer, Berlin Heidelberg New York 1975.

\bibitem{Toda} Toda, H. Composition methods in homotopy groups of spheres, Ann. of Math. Studies 49, 1962.

\bibitem{Wall1965} Wall, C. T. C.
Finiteness conditions for CW-complexes.
Ann. of Math. (2) 81 (1965), 56-69.

\bibitem{WhiteGW1978} {Whitehead}, G. W. Elements of homotopy theory, Grad. Texts in Math., 61, Springer-Verlag, New York, 1978.

\bibitem{WhiteheadJHC} Whitehead, J. H. C.
On adding relations to homotopy groups.
Ann. of Math. (2) 42 (1941), 409-428.

\end{thebibliography}
\end{document}